\documentclass[12pt,reqno]{amsart}
\usepackage{etoolbox}

\makeatletter
\patchcmd\maketitle
{\uppercasenonmath\shorttitle}
{}
{}{}
\patchcmd\maketitle
{\@nx\MakeUppercase{\the\toks@}}
{\the\toks@}
{}
{}{}
\patchcmd\@settitle{\uppercasenonmath\@title}{\Large}{}{}
\patchcmd\@setauthors
{\MakeUppercase{\authors}}
{\authors}
{}{}
\makeatother
\usepackage{amsmath,amssymb,amsthm, amsfonts}
\usepackage{enumitem}
\usepackage{color}
\usepackage{url}
\usepackage{tikz-cd}
\usepackage{combelow}
\usepackage{bm}
\input{mathrsfs.sty}
\usepackage[utf8]{inputenc}
\usepackage[T1]{fontenc}
\newtheorem{theorem}{Theorem}[section]

\newtheorem{definition}{Definition}[section]

\newtheorem{corollary}{Corollary}[section]

\newtheorem{lemma}{Lemma}[section]
\newtheorem{remark}{Remark}[section]
\newtheorem{example}{Example}[section]
\numberwithin{equation}{section}

\newcommand{\vertiii}[1]{{\left\vert\kern-0.25ex\left\vert\kern-0.25ex\left\vert #1 
		\right\vert\kern-0.25ex\right\vert\kern-0.25ex\right\vert}}

\newcommand\norm[1]{\left\lVert#1\right\rVert}

\usepackage[colorlinks=true]{hyperref}
\hypersetup{urlcolor=blue, citecolor=red , linkcolor= blue}
\usepackage{cite}

\begin{document}
		\title[On inequalities involving the spherical operator transforms]{On inequalities involving the spherical operator transforms}
		

		\author[F. Kittaneh,  S. Sahoo and H. Stankovi\'c]{Fuad Kittaneh, Satyajit Sahoo and Hranislav Stankovi\'c}

		\address{(Kittaneh) Department of Mathematics, The University of Jordan, Amman, Jordan and Department of Mathematics, Korea University, Seoul 02841, South Korea}
		\email{\url{fkitt@ju.edu.jo}}
		
		\address{(Sahoo) Department of Mathematics, School of Basic Sciences, Indian Institute of Technology Bhubaneswar, Bhubaneswar, Odisha 752050, India}
		\email{\url{ssahoomath@gmail.com}}
		
		\address{(Stankovi\'c) Faculty of Electronic Engineering, University of Ni\v s, Aleksandra Medvedeva 14, Ni\v s, Serbia
		}
		\email{\url{hranislav.stankovic@elfak.ni.ac.rs}}
		
		
		\date{\today}
		
		\maketitle

		\begin{abstract}
			This paper explores refinements of some operator norm inequalities through the generalized spherical Aluthge transform and the spherical Heinz transform. We introduce the spherical Schatten $p$-norm for operator tuples and establish several related inequalities. Additionally, equality conditions for some of these inequalities are also presented.  Furthermore, we define the (joint) Schatten $p$-numerical radius and the Schatten hypo-$p$-norm for operator tuples, deriving some fundamental inequalities in this setting.
		\end{abstract}
		
		\bigskip 
		\section{Introduction and preliminary results}

		Let $(\mathcal{H},\langle \cdot, \cdot \rangle)$ be a separable complex Hilbert space, and denote by $\mathfrak{B}(\mathcal{H})$ the space of all bounded linear operators on $\mathcal{H}$. The null space and the range of an operator $T \in \mathfrak{B}(\mathcal{H})$ are denoted by $\mathcal{N}(T)$ and $\mathcal{R}(T)$, respectively. The adjoint of $T$ is represented by $T^*$. The modulus of $T$ is given by $|T| = (T^*T)^{1/2}$, while its real and imaginary parts are defined as $\Re (T) = \frac{T + T^*}{2}$ and $\Im(T) = \frac{T - T^*}{2i}$, respectively.		
		An operator $T$ is called positive, denoted $T \geq 0$, if $\langle Tx, x \rangle \geq 0$ for all $x \in \mathcal{H}$. If $-T \geq 0$, we write $T \leq 0$. An operator $T$ is self-adjoint (or Hermitian) if $T = T^*$. The set of positive operators forms a convex cone in $\mathfrak{B}(\mathcal{H})$, inducing the partial order $\geq$ on the set of self-adjoint operators: for Hermitian operators $A$ and $B$, we write $A \geq B$ if and only if $A - B \geq 0$. This order is known as the L\"{o}wner order. It is evident that $|T| \geq 0$, and that $\mathrm{Re\,} T$ and $\mathrm{Im\,} T$ are self-adjoint for any operator $T \in \mathfrak{B}(\mathcal{H})$. 		
		An operator $T$ is normal if it satisfies $T^*T = TT^*$, and it is unitary if $T^*T = TT^* = I$. 
		Due to  importance of normal operators in operator theory, as well as in quantum mechanics, many generalizations of this class have appeared over the decades.  An operator $T$ is said to be quasinormal if $T$ commutes with $T^*T$, i.e., $TT^*T=T^*T^2$, and hyponormal if $TT^*\leq T^*T$. Clearly, both of these classes generalize the class of normal operators, and we have that
		\begin{align*}
			\text{normal}&\,\Rightarrow\,\text{quasinormal}\,\Rightarrow\,\text{hyponormal}.
		\end{align*}
		
		For other generalizations of normal operators, see, for example, \cite{Furuta01}.
		
		\medskip 
		
		Closely related to hyponormal operators (and especially \( p \)-hyponormal operators) are various operator transforms. If $T\in\mathfrak{B}(\mathcal{H}) $, then $T$ admits a polar decomposition $T=U|T|$, where $U$ is a partial isometry (so that $U^*U$ and $UU^*$ are projections).
        The Aluthge transform of an operator \( T \in \mathfrak{B}(\mathcal{H}) \) is given by $\widetilde{T} = |T|^{1/2} U |T|^{1/2}.$ The Duggal transform of \( T \) is defined as $T^D = |T|U,$ 	 
		while the mean transform is given by  $\widehat{T} = \frac{1}{2}(U|T| + |T|U) = \frac{1}{2}(T + T^D).$ 
		For \( t \in [0,1] \), the generalized Aluthge transform of \( T \) is defined as $\widetilde{T}(t) = |T|^t U |T|^{1-t}.$ 			Clearly,   
		$\widetilde{T}(0) = T$, $\widetilde{T}(1) = T^D,$  \text{and} $\widehat{T} = \frac{1}{2}(\widetilde{T}(0) + \widetilde{T}(1)).$ 		
		Recently, the authors of \cite{BenhidaChoKo20} introduced the generalized mean transform of \( T \), given by  
		\[
		\widehat{T}(t) = \frac{1}{2}(\widetilde{T}(t) + \widetilde{T}(1-t)),
		\]  
		while the $\lambda$-mean transform of $T$ is introduced in \cite{Zamani_JMAA_2021}, and it is defined by 
		$$M_\lambda(T)=\lambda T+(1-\lambda)T^D$$
		for $\lambda\in [0, 1]$. In particular, $M_0(T)=T^D$ and $M_{\frac{1}{2}}(T)=\widehat{T}$ is the mean transform of $T$. 				For further details on the Aluthge and Duggal transforms, see \cite{Aluthge90, ChoJungLee05, JungKoPearcy00}. In recent years, the mean transform and its generalizations have also received significant attention (see \cite{ACFS24a, ACFS24b, ChabbabiMbekhta17, ChabbabiOsterman22, LeeLeeYoon14}).

		\medskip

		Next, let $\mathfrak{K}(\mathcal{H})$ denote the ideal of compact operators on the Hilbert space $\mathcal{H}$.  For $T\in\mathfrak{K}(\mathcal{H})$ the singular values of $T$, denoted by $s_1(T), s_2(T), \dots$, correspond to the eigenvalues of the positive operator $|T|$ and are arranged in decreasing order, accounting for multiplicity. Furthermore, let 
		\begin{equation*}
			\mathcal{S}:=\left\{T\in\mathfrak{K}(\mathcal{H}):\, \sum_{j=1}^\infty s_j(T)<\infty
			\right\}.
		\end{equation*}
		Operators in $\mathcal{S}$ are called the trace class operators.
		The trace functional, denoted by $\text{tr}(\cdot)$, is defined on $\mathcal{S}$ as
		\begin{equation}\label{traza}
			\text{tr}(T) = \sum_{j=1}^{\infty} \langle Te_j, e_j\rangle,\quad T\in\mathcal{S},
		\end{equation}
		where $\{e_j\}_{j=1}^{\infty}$ forms an orthonormal basis for the Hilbert space $\mathcal{H}$. It is worth noting that this definition coincides with the standard trace definition when $\mathcal{H}$ is finite-dimensional. The series in \eqref{traza} converges absolutely, and its value remains unchanged regardless of the choice of basis.
		
		Additionally, let us clarify the definition of the Schatten $p$-class with $p \geq 1$. An operator $T$ belongs to the Schatten $p$-class, denoted as $\mathfrak{C}_p(\mathcal{H})$, if the sum of the $p$-th powers of its singular values is finite. More precisely, $T \in \mathfrak{C}_p(\mathcal{H})$ if 
		
		\[
		\text{tr}(|T|^p) = \sum_{j=1}^{\infty} s_j(T)^p < \infty.
		\]
		The Schatten $p$-norm of $T \in \mathfrak{C}_p(\mathcal{H})$ is given by
		\begin{equation*}
			\|T\|_p := \left[\text{tr}(|T|^p)\right]^\frac{1}{p}.
		\end{equation*}
		When $p=2$, the ideal $\mathfrak{C}_2(\mathcal{H})$ is referred to as the Hilbert--Schmidt class. In this case, $\mathfrak{C}_2(\mathcal{H})$ forms a Hilbert space with the inner product $\langle T, S\rangle_2 = \text{tr}(TS^*)$. Also, when $p=1$, we obviously have that $\mathfrak{C}_1(\mathcal{H})=\mathcal{S}$. 
		
		Recall that for any unitarily invariant norm $|||\cdot|||$ on $\mathcal{H}$, the invariance property $|||UTV||| = |||T|||$ holds for any pair of unitary operators $U$ and $V$, provided that $T \in \mathcal{J}_{|||\cdot|||}$, where $\mathcal{J}_{|||\cdot|||}$ denotes the norm ideal associated with $|||\cdot|||$. The standard operator norm and the $p$-Schatten norms are examples of unitarily invariant, or symmetric, norms. It is well known that $|||T||| = |||\,|T|\,||| = |||T^*|||$ for any $T \in \mathcal{J}_{|||\cdot|||}$. Furthermore, all these ideals are contained within the ideal of compact operators. For a comprehensive exploration of the general theory of unitarily invariant norms, we refer to \cite{GK}.

		\medskip

		Furthermore, let us introduce some notations from the multivariable operator theory, as well as briefly recall some classes of operator $d$-tuples, which will be of interest in the sequel. For $\mathbf{T}=(T_1,\ldots,T_d)\in\mathfrak{B}(\mathcal{H})^d$,
		by $\mathbf{T^*}$ we denote the operator $d$-tuple $\mathbf{T^*}=(T_1^*,\ldots,T_d^*)\in\mathfrak{B}(\mathcal{H})^d$. 		 If $\mathbf{T}=(T_1,\ldots, T_m)\in\mathfrak{B}(\mathcal{H})^m$ and $\mathbf{S}=(S_1,\ldots, S_n)\in\mathfrak{B}(\mathcal{H})^n$, the product of $\mathbf{T}$ and $\mathbf{S}$ is given by
		\begin{equation*}\label{eq:product_def}
			\mathbf{T}\circ\mathbf{S}:=(T_1S_1,\ldots,T_1S_n,\ldots, T_mS_1,\ldots, T_mS_n).
		\end{equation*}
		Related to this definition, we introduce the following notation:
		\begin{equation*}
			\mathbf{T}^1=\mathbf{T}\quad \text{ and }\quad \mathbf{T}^{n+1}=\mathbf{T}\circ\mathbf{T}^n,
		\end{equation*}
		where $n\in\mathbb{N}$. 
		
		For $S,T\in\mathfrak{B}(\mathcal{H})$, let $[S,T]=ST-TS$ be the commutator of $S$ and $T$. We say that an $d$-tuple $\mathbf{T}=(T_1,\ldots,T_d)\in\mathfrak{B}(\mathcal{H})^d$ of operators on $\mathcal{H}$ is jointly  hyponormal  if the operator matrix
		\begin{equation*}
			[\mathbf{T^*},\mathbf{T}]:=\begin{bmatrix}
				[T_1^*,T_1]&[T_2^*,T_1]&\cdots&[T_d^*,T_1]\\
				[T_1^*,T_2]&[T_2^*,T_2]&\cdots&[T_d^*,T_2]\\
				\vdots&\vdots&\ddots&\vdots\\
				[T_1^*,T_d]&[T_2^*,T_d]&\cdots&[T_d^*,T_d]
			\end{bmatrix}
		\end{equation*}
		is positive on the direct sum of $n$ copies of $\mathcal{H}$ (cf. \cite{Athavale88, CurtoMuhlyXia88}). An operator $d$-tuple $\mathbf{T}=(T_1,\ldots,T_d)\in\mathbb{B}(\mathcal{H})^d$ is said to be  commuting  (or a  $d$-tuple of commuting operators) if $T_iT_j=T_jT_i$ for all $i,j\in\{1,\ldots,d\}$. A $d$-tuple $\mathbf{T}$ is said to be  normal  if $\mathbf{T}$ is commuting  and each $T_k$ is normal, $k\in\{1,\ldots, d\}$. 
		The notion of quasinormality has  several multivariable analogues. We shall restrict ourselves to the spherically quasinormal operator tuples. We say that a commuting operator $d$-tuple is spherically quasinormal if each $T_i$ commutes with $\sum_{j=1}^{n}T_j^*T_j$, $i\in\{1,\ldots,d\}$ (see \cite{Gleason06}). As shown in \cite{Athavale88}, \cite{AthavalePodder15} and \cite{Gleason06}, we have
		
		\begin{align*}
			\text{normal}\,\Rightarrow\ \text{spherically quasinormal}\,\Rightarrow\, \text{jointly hyponormal}.
		\end{align*}
		
		For more details on these classes and the relations between them, see \cite{CurtoLeeYoon05, CurtoLeeYoon20, CvetkovicIlicStankovic24, Stankovic23_factors, Stankovic23_roots}.	
		The norm of an operator $d$-tuple appears in several important forms. 
		The  spherical operator norm (or joint operator norm) of $\mathbf T=(T_1,\ldots,T_d) \in \mathfrak B(\mathcal H)^{d}$ is given by
		$$\|\mathbf{T}\|:=\displaystyle\sup\left\{\left(\displaystyle\sum_{k=1}^d\|T_kx\|^2\right)^{\frac{1}{2}}:\;x\in \mathcal{H},\;\|x\|=1\right\}.$$
		The following simple observations provide useful characterizations of the spherical operator norm.
		\begin{lemma}\cite[Lemma 2.1]{FekiYamazaki21}\label{lem:norm_P}
			Let $\mathbf{T}=(T_1,\ldots, T_d)\in\mathfrak{B}(\mathcal{H})^d$. Then
			\begin{equation*}
				\norm{\mathbf{T}}=\norm{\sum_{k=1}^dT_k^*T_k}^{\frac{1}{2}}.
			\end{equation*}
		\end{lemma}
		
		\begin{lemma}\cite[Lemma 2.1]{AltwaijryDragomirFeki23}\label{lem:matrix_norm}
			Let $\mathbf{T}=(T_1,\ldots, T_d)\in\mathfrak{B}(\mathcal{H})^d$. Then
			\begin{equation*}
				\norm{\mathbf{T}}=\norm{\mathbb{T}},
			\end{equation*}
			where $\mathbb{T}$ denotes the following operator matrix on $\mathcal{H}_d=\oplus_{i=1}^d\mathcal{H}$:
			\begin{equation}\label{eq:op_matrix_def}
				\mathbb{T}=\begin{bmatrix}
					T_1&0&\cdots&0\\
					\vdots&\vdots&\phantom{asd}&\vdots\\
					T_d&0&\cdots&0
				\end{bmatrix}.
			\end{equation}
		\end{lemma}
		The hypo-norm introduced in \cite{dr2019} (as a special case of more general hypo-$p$-norms) is defined by
		\begin{equation*}
			\|\mathbf{T}\|_h :=  \sup_{(\lambda_1,\ldots,\lambda_d)\in {\overline{\mathbb{B}}_d}}\|\lambda_1T_1+\ldots+\lambda_dT_d\|,
		\end{equation*}
		where ${\mathbb{B}}_d$ denotes the open unit ball of $\mathbb{C}^d$ with respect to the Euclidean norm $\|\cdot\|_2,$ i.e.,
		\begin{equation}\label{eq:unit_ball}
			\mathbb{B}_d:=\left\{\bm{\lambda}=(\lambda_1,\ldots,\lambda_d)\in \mathbb{C}^d\,:\; \norm{\bm{\lambda}}_2^2=\sum_{k=1}^d|\lambda_k|^2\leq 1 \right\}.
		\end{equation}
		As shown recently, it turns out that $\|\mathbf{T}\|$ and $\|\mathbf{T}\|_h$ are always equivalent on $\mathfrak B(\mathcal H)^{d}$ (see \cite[Theorem 1.18]{Popescu09} and \cite[Proposition 2.1]{Feki20}):
		\begin{equation*}
			\frac{1}{\sqrt{d}}\|\mathbf T\| \leq \|\mathbf T\|_h \leq \|\mathbf T\|, \quad \mathbf T \in \mathfrak B(\mathcal H)^{d}.
		\end{equation*}
		There is also a notion of the Euclidean operator norm, which first appeared in \cite{Patel80} and it is given by
		\begin{equation*}
			\norm{\mathbf{T}}_e:=\left(\sum_{k=1}^{n}\norm{T_k}^2\right)^{1/2}.
		\end{equation*}
		However, when we speak about \emph{the}  norm of an operator $d$-tuple, we usually think of a spherical operator norm.  It is worth mentioning that while $\|\mathbf{T}\|_h=\|\mathbf{T}^*\|_h$ and $\|\mathbf{T}\|_e=\|\mathbf{T}^*\|_e$ for any $\mathbf{T}\in \mathfrak B(\mathcal H)^{d}$, this is not the case for $\|\cdot\|$. For other related concepts, we refer the reader to \cite{AltwaijryFekiStankovic25, Stankovic25}.

		For $\mathbf{T}=(T_1,\ldots,T_d)\in \mathfrak B(\mathcal H)^{d}$, we define the joint $($or spherical$)$ numerical radius of $\mathbf{T}$ as
		\begin{equation*}
			\omega(\mathbf{T})
			=\displaystyle\sup \left\{\Big(\displaystyle \sum_{k=1}^d|\langle T_kx, x\rangle|^2\Big)^{\frac{1}{2}}:\;x\in \mathcal{H},\;\|x\|=1\right\}.
		\end{equation*}
		For more details, see \cite{ChoTakaguchi81}. It was shown in \cite[Proof of Theorem 1.19]{Popescu09} that $\omega(\mathbf T)$ coincides with the  hypo-numerical radius $\omega_h(\mathbf T)$, which is given by
		\begin{equation*}
			\omega_h(\mathbf{T}):= \displaystyle\sup_{(\lambda_1,\ldots,\lambda_d)\in \mathbb{B}_d} \omega(\lambda_1T_1+\ldots+\lambda_dT_d).
		\end{equation*}
		
		For $\mathbf{T}= (T_1,\ldots,T_d)\in \mathfrak B(\mathcal H)^{d}$, it was shown in \cite[Theorem 2.2]{BakloutiFeki18} and \cite[Theorem 2.4]{BakloutiFekiAhmed18} that
		\begin{equation*}\label{eq:ffineq}
		\frac{1}{2\sqrt{d}}\|\mathbf{T}\| \leq\omega(\mathbf{T})\leq \|\mathbf{T}\|_e\leq \|\mathbf{T}\|.
		\end{equation*}

		The spectrum of an operator tuple also has
		many (nonequivalent) definitions. See, for instance, \cite{ArensCalderon55, Taylor70a}.
		An operator $d$-tuple $\mathbf{T}\in\mathfrak{B}(\mathcal{H})^d$ is said to be Taylor invertible if its associated Koszul complex $\mathcal{K}(\mathbf{T},\mathcal{H})$ is exact. For $d=2$, the Koszul complex $\mathcal{K}(\mathbf{T},\mathcal{H})$ associated with $\mathbf{T}=(T_1,T_2)$ on $\mathcal{H}$ is given by:
		\begin{equation*}
			\mathcal{K}(\mathbf{T},\mathcal{H}):\quad 0\longrightarrow\mathcal{H}\stackrel{\mathbf{T}}{\longrightarrow}\mathcal{H}\oplus\mathcal{H}\stackrel{(-T_2 \,\,T_1)}{\longrightarrow}\mathcal{H}\longrightarrow0,
		\end{equation*}
		where $\mathbf{T}=\begin{pmatrix}
			T_1\\T_2
		\end{pmatrix}$.

		Let $d\in\mathbb{N}$.	For $T_1,\ldots, T_d\in\mathfrak{B}(\mathcal{H})$, consider a $d$-tuple $\mathbf{T}=\begin{pmatrix}
			T_1\\\vdots\\T_d
		\end{pmatrix}$ as an operator from $\mathcal{H}$ into $\mathcal{H}\oplus\cdots\oplus\mathcal{H}$, that is,
		\begin{equation}\label{eq:operator_column}
			\mathbf{T}=\begin{pmatrix}
				T_1\\\vdots\\T_d
			\end{pmatrix}:\,\mathcal{H}\to \begin{array}{c}
				\mathcal{H}\\\oplus\\\vdots\\\oplus\\\mathcal{H}
			\end{array}.
		\end{equation}
		We define (canonical) spherical polar decomposition of $\mathbf{T}$ (cf. \cite{CurtoYoon16, CurtoYoon18, KimYoon17}) as
		
		\begin{equation*}
			\mathbf{T}=\begin{pmatrix}
				T_1\\\vdots\\T_d
			\end{pmatrix}=\begin{pmatrix}
				V_1\\\vdots\\V_d
			\end{pmatrix}P=\begin{pmatrix}
				V_1P\\\vdots\\V_dP
			\end{pmatrix}=\mathbf{V}P,
		\end{equation*}
		where $P:=\sqrt{T_1^*T_1+\cdots+T_d^*T_d}$ is a positive operator on $\mathcal{H}$, and \begin{equation*}
			\mathbf{V}=\begin{pmatrix}
				V_1\\\vdots\\V_d
			\end{pmatrix}:\,\mathcal{H}\to \begin{array}{c}
				\mathcal{H}\\\oplus\\\vdots\\\oplus\\\mathcal{H}
			\end{array},
		\end{equation*}
		is a spherical partial isometry from $\mathcal{H}$ into $\mathcal{H}\oplus\ldots\oplus\mathcal{H}$. In other words, $V_1^*V_1+\cdots+V_d^*V_d$ is the (orthogonal) projection onto the initial space of the partial isometry $\mathbf{V}$ which is 
		\begin{align*}
			\mathcal{N}(\mathbf{T})^\perp&=\left(\bigcap_{i=1}^d\mathcal{N}(T_i)\right)^\perp=\mathcal{N}(P)^\perp=\left(\bigcap_{i=1}^d\mathcal{N}(V_i)\right)^\perp.
		\end{align*}

		The following characterization of spherically quasinormal tuples in terms of the polar decomposition will be useful later.
		\begin{theorem}\label{thm:spher_quasi_matrix}\cite[Theorem 1.2]{Stanković_2024}
			Let $\mathbf{T}=(T_1,\ldots, T_d)\in\mathfrak{B}(\mathcal{H})^d$ be a $d$-tuple of commuting operators. The following conditions are equivalent:
			\begin{enumerate}[label=\textit{(\roman*)}]
				\item $\mathbf{T}$ is spherically quasinormal;
				\item $\mathbb{P}\mathbb{V}=\mathbb{V}\mathbb{P}$,
				where \begin{equation*}
					\mathbb{V}=\begin{bmatrix}
						V_1&0&\cdots&0\\
						\vdots&\vdots&\ddots&\vdots\\
						V_d&0&\cdots&0
					\end{bmatrix}
					\quad \text{ and }\quad 
					\mathbb{P}=\begin{bmatrix}
						P & \dots & 0\\
						\vdots  &\ddots &\vdots\\
						0 & \dots &  P
					\end{bmatrix}.
				\end{equation*}
			\end{enumerate}
		\end{theorem}

		Next, let us recall the multivariable analogues of the earlier mentioned operator transforms. For $\mathbf{T}=\mathbf{V}P=(V_1P,\ldots, V_dP)\in\mathfrak{B}(\mathcal{H})^d$, the following  concepts are introduced:
		\begin{itemize}
			\item spherical Aluthge transform: 
			$$\widetilde{\mathbf{T}}=(\widetilde{\mathbf{T}}_1,\ldots,\widetilde{\mathbf{T}}_d)=(P^{\frac{1}{2}}V_1P^{\frac{1}{2}},\ldots,P^{\frac{1}{2}}V_dP^{\frac{1}{2}});$$
			
			\item spherical Duggal transform: 
			$$\mathbf{T}^D=(\mathbf{T}^D_1,\ldots,\mathbf{T}^D_d)=(PV_1,\ldots,PV_d);$$		
			
			\item  spherical mean transform: 
			$$\widehat{\mathbf{T}}=\dfrac{1}{2}(\mathbf{T}+\mathbf{T}^D);$$
			
			\item generalized spherical Aluthge transform: 
			
			\begin{equation*}
				\widetilde{\mathbf{T}}(t)=(\widetilde{\mathbf{T}}_1(t),\ldots,\widetilde{\mathbf{T}}_d(t))=(P^{t}V_1P^{1-t},\ldots,P^{t}V_dP^{1-t}),
			\end{equation*}
			for $t\in[0,1]$;
			
			\item spherical Heinz transform: 
			\begin{equation*}\label{definition_eq}
				\widehat{\mathbf{T}}(t)=\dfrac{1}{2}(\widetilde{\mathbf{T}}(t)+\widetilde{\mathbf{T}}(1-t)),
			\end{equation*}
			for $t\in[0,1]$.	In other words, $\widehat{\mathbf{T}}(t)=(\widehat{\mathbf{T}}_1(t),\ldots,\widehat{\mathbf{T}}_d(t))$, where
			\begin{equation*}
				\widehat{\mathbf{T}}_i(t)=\dfrac{1}{2}\left(P^tV_iP^{1-t}+P^{1-t}V_iP^t\right),\quad i\in\{1,\ldots,d\}.
			\end{equation*}
			\noindent Obviously, $\widehat{\mathbf{T}}(t)=\widehat{\mathbf{T}}(1-t)$, $\widehat{\mathbf{T}}(0)=\widehat{\mathbf{T}}(1)=\widehat{\mathbf{T}}$ and $\widehat{\mathbf{T}}(\frac{1}{2})=\widetilde{\mathbf{T}}$.
		\end{itemize}

		For more details on the mentioned concepts, we refer the reader to \cite{BenhidaCurtoLee22, CurtoYoon16, CurtoYoon18, FekiYamazaki21, KimYoon17, Stankovic24a, Stanković_2024}. 
				Finally, we can extend the notion of the $\lambda$-mean transform to a multivariable setting,
		as well.
		\begin{definition}
			Let $\mathbf{T}=\mathbf{V}P=(V_1P,\ldots, V_dP)\in\mathfrak{B}(\mathcal{H})^d$ be the canonical spherical polar
			decomposition of a $d$-tuple $\mathbf{T}$ and let $\lambda\in [0, 1]$. The spherical $\lambda$-mean transform of $\mathbf{T}$ is defined as
			$$\mathbf{M}_\lambda(\mathbf{T})=(\mathbf{M}_\lambda^1(\mathbf{T}),\ldots, \mathbf{M}_\lambda^d(\mathbf{T}))=\lambda \mathbf{T}+(1-\lambda)\mathbf{T}^D.$$
			In particular, $\mathbf{M}_0(\mathbf{T})=\mathbf{T}^D$ and $\mathbf{M}_{\frac{1}{2}}(\mathbf{T})=\widehat{\mathbf{T}}$.
		\end{definition}
		Note that many properties of the ordinary spherical mean transform (see \cite{Stankovic24a}) also hold for the spherical $\lambda$-mean transform.

		\medskip 
		
		The organization of the paper is as follows.  Section 2 presents several refinements of operator norm inequalities  using the generalized spherical Aluthge transform and the spherical Heinz transform.  Section 3 introduces the spherical Schatten $p$-norm for operator tuples and establishes related inequalities involving these transforms. Additionally, equality conditions for some of these inequalities are also discussed.  
		In Section 4, we define  the (joint) Schatten $p$-numerical radius and the Schatten hypo-$p$-norm for operator tuples, deriving several fundamental inequalities in this context.

		\bigskip
		
		\section{Operator norm inequalities}
		In this section, we have established some refinements of earlier results.\\
		For every unitarily invariant norm, we have the Heinz
		inequalities
		\begin{align}\label{Ineq_Heinz_0}
			2\vertiii{A^\frac{1}{2}XB^\frac{1}{2}}\leq \vertiii{A^\nu X B^{1-\nu}+A^{1-\nu}X B^\nu}\leq \vertiii{AX+XB},
		\end{align}
		for $A,B$, and $X$ be operators on a complex separable Hilbert
		space such that $A$ and $B$ are positive and for $0\leq \nu\leq 1$.
		
		The following inequality is a refinement of the second inequality in \eqref{Ineq_Heinz_0}, which follows from \cite[Corollary 3]{Kittaneh_IEOT_2010}. 
		Let $A, B$, and $X$ be operators such that $A$ and $B$ are positive.
		Then for $0 \leq\nu \leq 1$ and for every unitarily invariant norm,
		 \begin{equation}\label{Ref_Ineq_Heinz}
         \begin{split}
			\vertiii{A^\nu XB^{1-\nu}+A^{1-\nu} X B^\nu}&\leq 4r_0 \vertiii{A^\frac{1}{2}XB^\frac{1}{2}}+(1-2 r_0)\vertiii{AX+XB}\\&\leq \vertiii{AX+XB},
            \end{split}
		\end{equation}
		where $r_0=\min\{\nu, 1-\nu\}$.\\
		In \cite[Theorem 2]{Kittaneh_LMS_1993}, it was shown that for $A, B$ are positive operators in $\mathcal{B(H)}$ and $X\in \mathcal{B(H)}$,
		the inequality
		\begin{align}\label{Heinz_ineq_0}
			\vertiii{A^\nu XB^{1-\nu}}\leq \vertiii{AX}^\nu\vertiii{XB}^{1-\nu},
		\end{align}
		holds for every $\nu\in [0, 1]$.
  
		The following two results are refinements of \cite[Theorem 4.4]{Stanković_2024}.
		\begin{theorem}\label{thm:r0_norm_ineq}
			Let ${\bf T}=(T_1, \dots, T_d)\in \mathfrak{B}(\mathcal{H})^d$. Then 
			\begin{align*}
				\|\widehat{{\bf T}}(t)\|\leq 2r_0\|\widetilde{{\bf T}}\|+(1-2r_0)\|\widehat{{\bf T}}\|\leq \|\widehat{{\bf T}}\|\leq\|{\bf T}\|
			\end{align*}
			for any $t\in [0, 1]$, where $r_0=\min\{t, 1-t\}$.
			
			In particular,       $$\|\widetilde{{\bf T}}\|\leq\|\widehat{{\bf T}}\|\leq \|{\bf T}\|.$$
			
		\end{theorem}
		\begin{proof}
			Let $t\in [0, 1]$ be arbitrary and ${\bf T}={\bf V}P=(V_1P, \dots , V_d P)$ be the spherical polar decomposition of ${\bf T}$.
			Let $\mathbb{V}$  and $\mathbb{P}$ be as in Theorem \ref{thm:spher_quasi_matrix}. By Lemma \ref{lem:matrix_norm},
			\begin{equation}\label{Eq_00}
				\begin{split}
					\|\widehat{{\bf T}}(t)\|&=\frac{1}{2}\left\|\begin{bmatrix}
						P^tV_1P^{1-t}+P^{1-t}V_1P^t & 0& \dots& 0\\
						\vdots & \vdots && \vdots\\
						P^tV_dP^{1-t}+P^{1-t}V_dP^t & 0& \dots& 0
					\end{bmatrix}\right\| \\
					&=\frac{1}{2}\|\mathbb{P}^t\mathbb{V}\mathbb{P}^{1-t}+\mathbb{P}^{1-t}\mathbb{V}\mathbb{P}^t\|.
				\end{split}
			\end{equation}
			Now, using \eqref{Ref_Ineq_Heinz}, we have
			\begin{align}\label{Eq_0}
				\|\widehat{{\bf T}}(t)\|&\leq 2r_0 \|\mathbb{P}^\frac{1}{2}\mathbb{V}\mathbb{P}^\frac{1}{2}\|+\frac{(1-2 r_0)}{2}\|\mathbb{PV}+\mathbb{VP}\|\nonumber\\
				&\leq \frac{1}{2}\|\mathbb{PV}+\mathbb{VP}\|=\|\widehat{{\bf T}}\|.
			\end{align}
			Therefore, 
			\begin{align}\label{Eq_1}
				2r_0 \|\mathbb{P}^\frac{1}{2}\mathbb{V}\mathbb{P}^\frac{1}{2}\|+\frac{(1-2 r_0)}{2}\|\mathbb{PV}+\mathbb{VP}\| = 2r_0\|\widetilde{{\bf T}}\|+(1-2r_0)\|\widehat{{\bf T}}\|.
			\end{align}
			Using \eqref{Eq_0} and \eqref{Eq_1}, we have
			\begin{align*}
				\|\widehat{{\bf T}}(t)\|\leq 2r_0\|\widetilde{{\bf T}}\|+(1-2r_0)\|\widehat{{\bf T}}\|\leq \|\widehat{{\bf T}}\|.
			\end{align*}
			Now using the inequality of \cite[Theorem 4.4]{Stanković_2024}, we get
			\begin{align*}
				\|\widehat{{\bf T}}(t)\|\leq 2r_0\|\widetilde{{\bf T}}\|+(1-2r_0)\|\widehat{{\bf T}}\|\leq \|\widehat{{\bf T}}\|\leq\|{\bf T}\|.
			\end{align*}
		\end{proof}

		\begin{theorem}\label{thm:ineq_td_sum}
			Let ${\bf T}=(T_1, \dots, T_d)\in \mathfrak{B}(\mathcal{H})^d$. Then 
			\begin{align}\label{eq:ineq_td_sum}
				\|\widetilde{{\bf T}}\|\leq \|\widehat{{\bf T}}(t)\|\leq \frac{1}{2}(\|{\bf T}^D\|^t\|{\bf T}\|^{1-t}+\|{\bf T}^D\|^{1-t}\|{\bf T}\|^t)\leq \|{\bf T}\|
			\end{align}
			for any $t\in [0, 1]$.
			
			In particular,         $$\|\widetilde{{\bf T}}\|\leq \|\widehat{{\bf T}}(t)\|\leq\frac{1}{2}( \|{\bf T}^D\|+\|{\bf T}\|)\leq \|{\bf T}\|.$$
		\end{theorem}
		\begin{proof}
			Let $t\in [0, 1]$ be arbitrary and ${\bf T}={\bf V}P=(V_1P, \dots , V_d P)$ be the spherical polar decomposition of ${\bf T}$.
			Using \eqref{Eq_00} and the triangle inequality, we have that
			\begin{align*}
				\|\widehat{{\bf T}}(t)\|
				&=\frac{1}{2}\|\mathbb{P}^t\mathbb{V}\mathbb{P}^{1-t}+\mathbb{P}^{1-t}\mathbb{V}\mathbb{P}^t\|\\
				&\leq \frac{1}{2}\|\mathbb{P}^t\mathbb{V}\mathbb{P}^{1-t}\|+\|\mathbb{P}^{1-t}\mathbb{V}\mathbb{P}^t\|.
			\end{align*}
			Now, using \eqref{Heinz_ineq_0}, we have
			\begin{align}\label{Eq_10}
				\|\widehat{{\bf T}}(t)\|&\leq  \frac{1}{2}\|\mathbb{PV}\|^{t}\|\mathbb{VP}\|^{1-t}+\frac{1}{2}\|\mathbb{PV}\|^{1-t}\|\mathbb{VP}\|^{t}\nonumber\\
				&= \frac{1}{2}(\|{\bf T}^D\|^t\|{\bf T}\|^{1-t}+\|{\bf T}^D\|^{1-t}\|{\bf T}\|^{t}).
			\end{align}
			Now, 
			\begin{align}\label{Eq_11}
				\|{\bf T}^D\|=\|\mathbb{PV}\|\leq \|\mathbb{V}\|\|\mathbb{P}\|=\|{\bf T}\|\left\|\sum_{i=1}^dV_i^*V_i\right\|^{\frac{1}{2}}\leq \|{\bf T}\|.
			\end{align}
			Using \eqref{Eq_10} and \eqref{Eq_11}, we have
			\begin{align}\label{EQ_11}
				\|\widehat{{\bf T}}(t)\|
				&\leq \frac{1}{2}(\|{\bf T}^D\|^t\|{\bf T}\|^{1-t}+\|{\bf T}^D\|^{1-t}\|{\bf T}\|^{t})\leq \|{\bf T}\|.
			\end{align}
			Let $A=B=\mathbb{P}, X=\mathbb{V}$ in the first inequality of \eqref{Ineq_Heinz_0}, we obtain
			\begin{align}\label{EQ_12}
				\|\widetilde{{\bf T}}\|=\|\mathbb{P}^{\frac{1}{2}}\mathbb{V}\mathbb{P}^{\frac{1}{2}}\|\leq \frac{1}{2}\|\mathbb{P}^t\mathbb{V}\mathbb{P}^{1-t}+\mathbb{P}^{1-t}\mathbb{V}\mathbb{P}^t\|=\|\widehat{{\bf T}}(t)\|
			\end{align}
			Now combining \eqref{EQ_11} and \eqref{EQ_12}, we obtain
			\begin{align*}
				\|\widetilde{{\bf T}}\|\leq \|\widehat{{\bf T}}(t)\|\leq \frac{1}{2}(\|{\bf T}^D\|^t\|{\bf T}\|^{1-t}+\|{\bf T}^D\|^{1-t}\|{\bf T}\|^t)\leq \|{\bf T}\|.
			\end{align*}
		\end{proof}

		Inspired by \cite[Theorem 3.2]{Zamani_JMAA_2021}, we extend the one-dimensional case to a multivariable operator setting. More precisely, 
		the following theorem is a generalization of \cite[Eq. (4.5)]{Stanković_2024}. (We emphasize that in \cite{Stanković_2024}, the hypo-norm of a tuple $\mathbf{T}$  is denoted by $\| \mathbf{T}\|_e$, which is here reserved for the Euclidean norm of $\mathbf{T}$. We also remark that there is a typo in \cite[Eq. (4.4)]{Stanković_2024} and \cite[Eq. (4.5)]{Stanković_2024}, as the right hand side of the second inequality in both equations should be  $\norm{\mathbf{T}}$ instead of $\norm{\mathbf{T}}_e$.)

			\begin{theorem}\label{thm:lambda_mean} 
				Let ${\bf T}=(T_1, \dots, T_d)\in \mathfrak{B}(\mathcal{H})^d$. Then  for each $0\leq \lambda\leq 1$,
				\begin{align}\label{Eq_2.18}
					2\sqrt{\lambda-\lambda^2}\|\widetilde{{\bf T}}\|_h \leq \|\mathbf{M}_\lambda(\mathbf{T})\|_h\leq \lambda \|{\bf T}\|_h+(1-\lambda)\|{\bf T}^D\|_h \leq  \|{\bf T}\|.
				\end{align}
				In particular,
				for $\lambda=\frac{1}{2}$,
				\begin{align*}
					\|\widetilde{{\bf T}}\|_h \leq \|\widehat{{\bf T}}\|_h \leq \|{\bf T}\|.
				\end{align*}
			\end{theorem}
			\begin{proof}
				Let ${\bf T}={\bf V}P=(V_1P, \dots , V_d P)$ be the spherical polar decomposition of ${\bf T}$. It follows from \eqref{Ineq_Heinz_0} that 
				\begin{align}\label{Eq_2.19}
					\|A^\frac{1}{2}XB^\frac{1}{2}\|\leq \left\|\frac{AX+XB}{2}\right\|,
				\end{align}
				for positive operators $A, B$, and $X\in \mathfrak{B}(\mathcal{H})$.
				Let $\bm{\mu}=(\mu_1,\ldots,\mu_d)\in\overline{\mathbb{B}}_d$ be arbitrary and let $U_{\bm{\mu}}=\sum_{i=1}^d\mu_iV_i$. Setting $A=(1-\lambda)P$, $X=U_{\bm{\mu}}$ and $B=\lambda P$ in \eqref{Eq_2.19}, we obtain
				\begin{align}\label{Eq2.20}
				2	\sqrt{\lambda-\lambda^2}\|P^{\frac{1}{2}}U_{\bm{\mu}} P^{\frac{1}{2}}\|\leq \|(1-\lambda)PU_{\bm{\mu}}+\lambda U_{\bm{\mu}} P\|.
				\end{align}
				Next, note that 
				\begin{align*}
					\mu_1 \widetilde{\mathbf{T}}_1+\cdots+\mu_d \widetilde{\mathbf{T}}_d&=P^{\frac{1}{2}}U_{\bm{\mu}} P^{\frac{1}{2}},\\
					\mu_1\mathbf{M}_\lambda^1(\mathbf{T})+\cdots+\mu_d\mathbf{M}_\lambda^d(\mathbf{T})&=(1-\lambda)PU_{\bm{\mu}}+\lambda U_{\bm{\mu}} P.
				\end{align*}
				By taking the supremum over all $(\mu_1,\ldots,\mu_d)\in\overline{\mathbb{B}}_d$ in \eqref{Eq2.20}, we have 
				\begin{align}\label{Eq2.21}
				2	\sqrt{\lambda-\lambda^2}\|\widetilde{{\bf T}}\|_h\leq  \|\mathbf{M}_\lambda(\mathbf{T})\|_h.
				\end{align}
				Further, since $\mathbf{M}_\lambda(\mathbf{T})=\lambda \mathbf{T}+(1-\lambda)\mathbf{T}^D$, and using the fact $\|{\bf T}^D\|_h\leq \|{\bf T}^D\|\leq \|{\bf T}\|$, we have
				\begin{align}\label{Eq2.22}
					\|\mathbf{M}_\lambda(\mathbf{T})\|_h\leq\lambda \|\mathbf{T}\|_h+(1-\lambda)\|\mathbf{T}^D\|_h \leq\|{\bf T}\|.
				\end{align}
				Now, \eqref{Eq_2.18} follows from \eqref{Eq2.21} and \eqref{Eq2.22}.
			\end{proof}
			
			Using a similar technique, we obtain the following theorem, which is a refinement of \cite[Theorem 4.5]{Stanković_2024}. 
			\begin{theorem}
			Let ${\bf T}=(T_1, \dots, T_d)\in \mathfrak{B}(\mathcal{H})^d$. Then 
			\begin{align*}
				\|\widehat{{\bf T}}(t)\|_h\leq 2r_0\|\widetilde{{\bf T}}\|_h+(1-2r_0)\|\widehat{{\bf T}}\|_h\leq \|\widehat{{\bf T}}\|_h\leq\|{\bf T}\|
			\end{align*}
			for any $t\in [0, 1]$, 	where $r_0=\min\{t, 1-t\}$.
			
			In particular,        $$\|\widetilde{{\bf T}}\|_h\leq\|\widehat{{\bf T}}\|_h\leq \|{\bf T}\|.$$
			
			\end{theorem}

 In a similar way as in Theorem \ref{thm:r0_norm_ineq} and Theorem \ref{thm:lambda_mean}, we can prove the following operator norm inequalities, which is a generalization of \cite[Eq. (4.2)]{Stanković_2024}. 
 \begin{theorem}
 	Let ${\bf T}=(T_1, \dots, T_d)\in \mathfrak{B}(\mathcal{H})^d$. Then  for each $0\leq \lambda\leq 1$,
 	\begin{align*}
 		2\sqrt{\lambda-\lambda^2}\|\widetilde{{\bf T}}\| \leq \|\mathbf{M}_\lambda(\mathbf{T})\|\leq \lambda \|{\bf T}\|+(1-\lambda)\|{\bf T}^D\| \leq  \|{\bf T}\|.
 	\end{align*}
 	In particular,
 	for $\lambda=\frac{1}{2}$,
 	\begin{align*}
 		\|\widetilde{{\bf T}}\| \leq \|\widehat{{\bf T}}\| \leq \|{\bf T}\|.
 	\end{align*}
 \end{theorem}

		The following theorem follows from the same method as used in \cite{Stanković_2024} and from \cite[Theorem 6]{OmarKittaneh_LAA_2019}.
		\begin{theorem}
			Let ${\bf T}=(T_1, \dots, T_d)\in \mathfrak{B}(\mathcal{H})^d$. Then 
			\begin{align*}
				\omega(\widetilde{{\bf T}}) \leq \omega(\widehat{{\bf T}}(t))\leq  \omega(\widehat{{\bf T}}),
			\end{align*}
			for any $t\in [0, 1]$.
		\end{theorem}

		\begin{theorem}
			Let ${\bf T}=(T_1, \dots, T_d)\in \mathfrak{B}(\mathcal{H})^d$. The following conditions are equivalent:
			\begin{enumerate}[label=\textit{(\roman*)}]
				\item $\mathbf{T^2}=\mathbf{0}$;
				\item $\widetilde{{\bf T}}(t)=\mathbf{0}$ for each $t\in (0,1]$;
				\item $\widetilde{{\bf T}}(t)=\mathbf{0}$ for some $t\in (0,1]$;
				\item $\widehat{{\bf T}}(t)=\mathbf{0}$ for each $t\in (0,1)$;
				\item $\widehat{{\bf T}}(t)=\mathbf{0}$ for some $t\in (0,1)$.
			\end{enumerate}
		\end{theorem}
		
		\begin{proof}
			The equivalence $(i)\Leftrightarrow(ii)\Leftrightarrow(iii)$ follows directly from \cite[Theorem 2.7]{Stanković_2024}. The implications $(ii)\Rightarrow(iv)\Rightarrow(v)$ are also obvious. Now assume that $\widehat{{\bf T}}(t)=\mathbf{0}$ for some $t\in (0,1)$. By \eqref{eq:ineq_td_sum}, we have that
			\begin{equation*}
				\|\widetilde{{\bf T}}\|\leq \|\widehat{{\bf T}}(t)\|=0,
			\end{equation*}
			i.e., $\widetilde{{\bf T}}=0$, showing that $(iii)$ holds. 
		\end{proof}
		
		\begin{remark}
			Note that the case $t=0$ in the previous theorem must be considered separately. In fact, it was shown in \cite[Corollary 2.3]{Stankovic24a} that
			\begin{equation*}
				\mathbf{T}=\mathbf{0}\quad\Longleftrightarrow\quad \widehat{{\bf T}}=\mathbf{0}. 
			\end{equation*}  
			We also mention that $\lambda=0$ in \cite[Proposition 3.2]{ACFS24a} should be excluded. 
		\end{remark}

		\bigskip 
		
		\section{Spherical Schatten $p$-norm inequalities}

		Let ${\bf T}=(T_1, \dots, T_d)\in \mathfrak{B}(\mathcal{H})^d$ and $1\leq p<\infty$. By treating $\mathbf{T}$ as an operator column given by \eqref{eq:operator_column},
		we can naturally say that the operator $d$-tuple $\mathbf{T}$ belongs to the Schatten $p$-class if $\begin{pmatrix}
			T_1\\\vdots\\T_d
		\end{pmatrix}\in\mathfrak{C}_p(\mathcal{H},\mathcal{H}^d)$, i.e., if
		\begin{equation*}
			\mathrm{tr}\,(P^p)<\infty,
		\end{equation*}
		where $P=\sqrt{T_1^*T_1+\cdots+T_d^*T_d}$. We shall simply write $\mathbf{T}\in \mathfrak{C}_p^d(\mathcal{H})$.
		
		\begin{definition}
			Let $\mathbf{T}\in \mathfrak{C}_p^d(\mathcal{H})$. The (spherical) Schatten $p$-norm of $\mathbf{T}$, $\norm{\mathbf{T}}_{s,p}$, is defined as
			\begin{equation*}
				\norm{\mathbf{T}}_{s,p}:=\left[\mathrm{tr}\,(P^p)\right]^\frac{1}{p}.
			\end{equation*}
		\end{definition}

		The following elementary lemma provides another way of looking at the quantity $\norm{\mathbf{T}}_{s,p}$.
		\begin{lemma}\label{lem:p_norm_equiv}
			Let ${\bf T}=(T_1, \dots, T_d)\in  \mathfrak{C}_p^d(\mathcal{H})$ and $1\leq p<\infty$. Then $\mathbb{T}\in  \mathfrak{C}_p(\mathcal{H}^d)$ and
			\begin{equation*} 
				\norm{\mathbf{T}}_{s,p}=\norm{P}_p=\norm{\mathbb{T}}_p,
			\end{equation*}
			where $\mathbb{T}$ is given by \eqref{eq:op_matrix_def}. 
		\end{lemma}
		
		\begin{proof}
			By direct computation, we have
			\begin{align*}
				\norm{\mathbb{T}}_p&=\norm{\begin{bmatrix}
						T_1&0&\cdots&0\\
						\vdots&\vdots&\phantom{asd}&\vdots\\
						T_d&0&\cdots&0
				\end{bmatrix}}_p\\&=\left[\mathrm{tr}\,\left(\begin{bmatrix}
					\sqrt{T_1^*T_1+\cdots+T_d^*T_d}&&&\\
					&0&&\\&&\ddots&\\&&&0
				\end{bmatrix}^p\right)\right]^\frac{1}{p}\\
				&=\left[\mathrm{tr}\,(P^p)\right]^\frac{1}{p}=\norm{\mathbf{T}}_p.
			\end{align*}
			The first equality is obvious.
		\end{proof}

		We will also need the following simple result.
		\begin{lemma}\label{lem:p_direct_sum}
			Let $1\leq p<\infty$ and $T_1,\ldots,T_d\in \mathfrak{C}_p(\mathcal{H})$. Then
			\begin{equation*}
				\norm{T_1\oplus\cdots \oplus T_d}_p=\left(\sum_{i=1}^d\norm{T_i}_p^p\right)^\frac{1}{p}
			\end{equation*}
		\end{lemma}
		
		\begin{proof}
			By the definition of $p$-norm, 
			\begin{align*}
				\norm{T_1\oplus\cdots \oplus T_d}_p&=\norm{\begin{bmatrix}
						T_1&&\\
						&\ddots&\\
						&&T_d
				\end{bmatrix}}_p=\left[\mathrm{tr}\,\left(\begin{bmatrix}
					|T_1|^p&&\\
					&\ddots&\\
					&&|T_d|^p
				\end{bmatrix}\right)\right]^\frac{1}{p}\\
				&=\left(\sum_{i=1}^d\mathrm{tr}\,(|T_i|^p)\right)^\frac{1}{p}=\left(\sum_{i=1}^d\norm{T_i}_p^p\right)^\frac{1}{p}.
			\end{align*}
		\end{proof}

		Note that if $\mathbf{T}\in \mathfrak{C}_p^d(\mathcal{H})$, then in particular, the operator column given in \eqref{eq:operator_column} is compact, and thus $T_i$ is compact for each $i\in\{1,\ldots,d\}$. A natural question arises whether each coordinate operator belongs to $\mathfrak{C}_p(\mathcal{H})$. The following theorem provides an affirmative answer to this question.

		\begin{theorem}\label{thm:coordinate_p_1_2}
			Let $1\leq p<\infty$ and ${\bf T}=(T_1, \dots, T_d)=(V_1P,\ldots, V_dP)\in \mathfrak{C}^d_p(\mathcal{H})$. Then $T_i\in \mathfrak{C}_p(\mathcal{H})$ and 
			$$\norm{T_i}_p\leq \norm{\mathbf{T}}_{s,p}$$
			for each $i\in\{1,\ldots,d\}$. 
		\end{theorem}
		
		\begin{proof}
			Let $i\in\{1,\ldots,d\}$ be arbitrary.  Since $V_i^*V_i\leq\sum_{i=1}^dV_i^*V_i\leq I$, we have that $\norm{V_i}\leq 1$, and thus
			\begin{equation*}
				\norm{T_i}_p=\norm{V_iP}_p\leq \norm{V_i}\norm{P}_p\leq \norm{P}_p.
			\end{equation*}
			The conclusion now follows from Lemma \ref{lem:p_norm_equiv}.
		\end{proof}

		\begin{theorem}\label{thm:m_lambda_p_ineq}
			Let $1\leq p<\infty$ and $\mathbf{T}=\mathbf{V}P=(V_1P,\ldots, V_dP)\in\mathfrak{C}^d_p(\mathcal{H})$. Then $\mathbf{M}_\lambda(\mathbf{T})\in\mathfrak{C}^d_p(\mathcal{H})$ for each $0\leq \lambda\leq 1$, and 
			\begin{equation}\label{eq:p_norm_main_ineq_1}
				\|\mathbf{M}_\lambda(\mathbf{T})\|_{s,p}\leq (\lambda +(1-\lambda)\sqrt[p]{d})\norm{\mathbf{T}}_{s,p}.
			\end{equation}
			In particular,
			\begin{equation}\label{eq:mean_p_ineq}
				\| \mathbf{T}^D\|_{s,p}\leq   \sqrt[p]{d} \norm{\mathbf{T}}_{s,p}\quad\text{ and }\quad 	\|\widehat{\mathbf{T}}\|_{s,p}\leq  \frac{1+\sqrt[p]{d}}{2}\norm{\mathbf{T}}_{s,p}.
			\end{equation}
		\end{theorem}				
		\begin{proof}
			Using the triangle inequality for the Schatten $p$-norm, and the fact that $\norm{TS}_p\leq \norm{T}_p\norm{S}$,  we have
			\begin{align*}
				\|\mathbf{M}_\lambda(\mathbf{T})\|_{s,p}&=\norm{\lambda \mathbf{T}+(1-\lambda)\mathbf{T}^D}_{s,p}\\
				&= \norm{\lambda\mathbb{V}\mathbb{P}+(1-\lambda)\mathbb{P}\mathbb{V}}_p\\
				&\leq \lambda\norm{\mathbb{V}\mathbb{P}}_p+(1-\lambda)\norm{\mathbb{P}\mathbb{V}}_p\\
				&\leq \lambda\norm{\mathbb{T}}_p+(1-\lambda)\norm{\mathbb{P}}_p\norm{\mathbb{V}}.
			\end{align*}
			Since $\sum_{i=1}^dV_i^*V_i$ is the orthogonal projection onto $\overline{\mathcal{R}(P)}$, Lemma \ref{lem:norm_P} and Lemma \ref{lem:matrix_norm} imply that
			\begin{equation*}
				\norm{\mathbb{V}}=\norm{\mathbf{V}}=\norm{\sum_{i=1}^dV_i^*V_i}^{\frac{1}{2}}\leq 1
			\end{equation*}
			Furthermore, by Lemma \ref{lem:p_norm_equiv} and Lemma \ref{lem:p_direct_sum}, we have that $\norm{\mathbb{T}}_p=\norm{\mathbf{T}}_{s,p}$, and 
			\begin{equation*}
				\norm{\mathbb{P}}_p=\norm{P\oplus\cdots \oplus P}_p=\left(\sum_{i=1}^d\norm{P}_p^p\right)^\frac{1}{p}=\sqrt[p]{d}\norm{P}_p=\sqrt[p]{d}\norm{\mathbf{T}}_{s,p}.
			\end{equation*}
			Therefore,
			\begin{align*}
				\|\mathbf{M}_\lambda(\mathbf{T})\|_{s,p}&\leq  \lambda\norm{\mathbb{T}}_p+(1-\lambda)\norm{\mathbb{P}}_p\norm{\mathbb{V}}\\&\leq \lambda\norm{\mathbf{T}}_{s,p}+(1-\lambda)\sqrt[p]{d}\norm{\mathbf{T}}_{s,p}\\&=(\lambda +(1-\lambda)\sqrt[p]{d})\norm{\mathbf{T}}_{s,p},
			\end{align*}
			showing that \eqref{eq:p_norm_main_ineq_1} holds. 
		\end{proof}

		In the case when $p=2$ and $\dim\,(\mathcal{H})<\infty$, we also have the following estimates.
		\begin{lemma}\label{lem:t_d_dim}
			Let $\mathbf{T}=\mathbf{V}P=(V_1P,\ldots, V_dP)$ be a $d$-tuple of $n\times n$ complex matrices. Then
			\begin{equation*}
				\norm{\mathbf{T}^D}_{s,2}\leq \sqrt{n}\norm{\mathbf{T}}_{s,2}.
			\end{equation*}
		\end{lemma}
		
		\begin{proof}
			Since $\mathbf{T}^D=(PV_1,\ldots,PV_d)$, we have that $(\mathbf{T}^D)^*\mathbf{T}^D=\sum_{i=1}^dV_i^*P^2V_i$, and thus
			\begin{align*}
				\norm{\mathbf{T}^D}_{s,2}^2&=\mathrm{tr} \left(\sum_{i=1}^dV_i^*P^2V_i\right)=\sum_{i=1}^d\mathrm{tr} \left(V_i^*P^2V_i\right)\\
				&=\sum_{i=1}^d\mathrm{tr} \left(P^2V_iV_i^*\right)=\mathrm{tr} \left(P^2\sum_{i=1}^dV_iV_i^*\right)
			\end{align*}
			Since $\mathrm{tr} (AB)\leq \mathrm{tr}(A)\mathrm{tr}\,(B)$ for positive semi-definite matrices $A$ and $B$, it follow that
			\begin{equation}\label{eq:td_ineq}
				\norm{\mathbf{T}^D}_{s,2}^2\leq \mathrm{tr} \left(P^2\right)\mathrm{tr}\,\left(\sum_{i=1}^dV_iV_i^*\right).
			\end{equation}
			Now, since $\sum_{i=1}^dV_i^*V_i\leq I$, using the monotonicity of the trace functional, we have
			\begin{align*}
				\mathrm{tr} \left(\sum_{i=1}^dV_iV_i^*\right)&=\sum_{i=1}^d\mathrm{tr} \left(V_iV_i^*\right)=\sum_{i=1}^d\mathrm{tr} \left(V_i^*V_i\right)\\
				&=\mathrm{tr} \left(\sum_{i=1}^dV_i^*V_i\right)\leq  \mathrm{tr} \left(I\right)=n.
			\end{align*}
			Furthermore, since $\norm{\mathbf{T}}_{s,2}^2=\mathrm{tr}\left(P^2\right)$, \eqref{eq:td_ineq} implies that
			\begin{equation*}
				\norm{\mathbf{T}^D}_{s,2}^2\leq n\norm{\mathbf{T}}_{s,2}^2,
			\end{equation*}
			and thus,
			\begin{equation*}
				\norm{\mathbf{T}^D}_{s,2}\leq \sqrt{n}\norm{\mathbf{T}}_{s,2}.
			\end{equation*}
		\end{proof}
		
		\begin{theorem}
			Let $\mathbf{T}=(T_1,\ldots,T_d)$ be a $d$-tuple of $n\times n$ complex matrices. Then
			\begin{equation*}
				\|\mathbf{M}_\lambda(\mathbf{T})\|_{s,2}\leq  (\lambda +(1-\lambda)\sqrt{\min\{n,d\}})\norm{\mathbf{T}}_{s,2}.
			\end{equation*}
			In particular,
			\begin{equation*}
				\|\widehat{\mathbf{T}}\|_{s,2}\leq  \frac{1+\sqrt{\min\{n,d\}}}{2}\norm{\mathbf{T}}_{s,2}.
			\end{equation*}
		\end{theorem}
		
		\begin{proof}
			It follows from Theorem \ref{lem:t_d_dim} that
			\begin{align*}
				\|\mathbf{M}_\lambda(\mathbf{T})\|_{s,2}&\leq \norm{\lambda \mathbf{T}+(1-\lambda)\mathbf{T}^D}_{s,2}\\&\leq \lambda \norm{\mathbf{T}}_{s,2}+(1-\lambda)\norm{\mathbf{T}^D}_{s,2} \\&\leq (\lambda +(1-\lambda)\sqrt{n})\norm{\mathbf{T}}_{s,2}.
			\end{align*}
			The conclusion now follows by combining the previous inequality with \eqref{eq:p_norm_main_ineq_1}.
		\end{proof}
		
		\begin{theorem}\label{thm:r0_p_ineq}
			Let $1\leq p<\infty$ and ${\bf T}=(T_1, \dots, T_d)\in \mathfrak{C}_p^d(\mathcal{H})$.  Then $\widehat{{\bf T}}(t)\in \mathfrak{C}^d_p(\mathcal{H})$ for each $0\leq t\leq 1$, and 
			\begin{align}\label{eq:r0_p_ineq}
				\|\widehat{{\bf T}}(t)\|_{s,p}\leq 2r_0\|\widetilde{{\bf T}}\|_{s,p}+(1-2r_0)\|\widehat{{\bf T}}\|_{s,p}\leq \|\widehat{{\bf T}}\|_{s,p}\leq\frac{1+\sqrt[p]{d}}{2}\norm{\mathbf{T}}_{s,p},
			\end{align} 
			where $r_0=\min\{t, 1-t\}$.
			
			In particular,        $$\|\widetilde{{\bf T}}\|_{s,p}\leq\|\widehat{{\bf T}}\|_{s,p}\leq\frac{1+\sqrt[p]{d}}{2} \|{\bf T}\|_{s,p}.$$
			
		\end{theorem}
		\begin{proof}
			Let $t\in [0, 1]$ be arbitrary and ${\bf T}={\bf V}P=(V_1P, \dots , V_d P)$ be the spherical polar decomposition of ${\bf T}$.
			Using Lemma \ref{lem:p_norm_equiv}, we have that
			\begin{equation}\label{eq:transform_t_p}
				\begin{split}
					\|\widehat{{\bf T}}(t)\|_{s,p}&=\frac{1}{2}\left\|\begin{bmatrix}
						P^tV_1P^{1-t}+P^{1-t}V_1P^t & 0& \dots& 0\\
						\vdots & \vdots && \vdots\\
						P^tV_dP^{1-t}+P^{1-t}V_dP^t & 0& \dots& 0
					\end{bmatrix}\right\|_p \\
					&=\frac{1}{2}\|\mathbb{P}^t\mathbb{V}\mathbb{P}^{1-t}+\mathbb{P}^{1-t}\mathbb{V}\mathbb{P}^t\|_p.
				\end{split}
			\end{equation}
			Now,  as in the proof of Theorem \ref{thm:r0_norm_ineq} we conclude that the first two inequalities hold \eqref{eq:r0_p_ineq}, while the last one follows from \eqref{eq:mean_p_ineq}.
		\end{proof}
		
		In a similar fashion, we can prove the following $p$-norm analogue of Theorem \ref{thm:ineq_td_sum}.
		\begin{theorem}\label{thm:ineq_td_sum_p}
			Let $1\leq p<\infty$ and ${\bf T}=(T_1, \dots, T_d)\in \mathfrak{C}_p^d(\mathcal{H})$. Then
			\begin{align*}\label{eq:ineq_td_sum_p}
				\|\widetilde{{\bf T}}\|_{s,p}\leq \|\widehat{{\bf T}}(t)\|_{s,p}\leq \frac{1}{2}(\|{\bf T}^D\|_{s,p}^t\|{\bf T}\|_{s,p}^{1-t}+\|{\bf T}^D\|_{s,p}^{1-t}\|{\bf T}\|_{s,p}^t)\leq \frac{d^\frac{t}{p}+d^\frac{1-t}{p}}{2}\|{\bf T}\|_{s,p}
			\end{align*}
			for any $t\in [0, 1]$.
			
			In particular,         $$\|\widetilde{{\bf T}}\|_{s,p}\leq \|\widehat{{\bf T}}(t)\|_{s,p}\leq\frac{1}{2}( \|{\bf T}^D\|_{s,p}+\|{\bf T}\|_{s,p})\leq \frac{1+\sqrt[p]{d}}{2} \|{\bf T}\|_{s,p}.$$
		\end{theorem}

		By combining Theorem \ref{thm:r0_p_ineq} and Theorem \ref{thm:ineq_td_sum_p}, we obtain the following corollary:
		\begin{corollary}
			Let $1\leq p<\infty$ and ${\bf T}=(T_1, \dots, T_d)\in \mathfrak{C}_p^d(\mathcal{H})$. Then
			\begin{equation}\label{eq:chain_ineq}
				\|\widetilde{{\bf T}}\|_{s,p}\leq \|\widehat{{\bf T}}(t)\|_{s,p}\leq \|\widehat{{\bf T}}\|_{s,p}\leq \frac{1+\sqrt[p]{d}}{2} \|{\bf T}\|_{s,p}.
			\end{equation}
			for any $t\in [0, 1]$.
		\end{corollary}

		\medskip

		Recall the following theorems from \cite{Kittaneh_IEOT_2010}, which consider the equality conditions in \eqref{Ineq_Heinz_0}.
		\begin{theorem}\cite[Theorem 5]{Kittaneh_IEOT_2010}\label{thm:equiv_cond_p}
			Let $A, B$, and $X$ be operators such that $A$ and $B$ are positive,
			and let $1 <p< \infty$. Then
			$$\|A^\nu X B^{1-\nu}+ A^{1-\nu} X B^\nu\|_p=\|AX+XB\|_p$$
			for some $\nu$ with $0<\nu <1$ if and only if $AX=XB$.
		\end{theorem}
		
		\begin{theorem}\cite[Theorem 6]{Kittaneh_IEOT_2010}\label{thm:equiv_cond_p_1}
			Let \( A, B, \) and \( X \) be operators such that \( A \) and \( B \) are positive and invertible, and let \( 1 < p < \infty \). Then
			\[
			2 \left\| A^{1/2} X B^{1/2} \right\|_p = \left\| A^\nu X B^{1-\nu} + A^{1-\nu} X B^\nu \right\|_p
			\]
			for some \( \nu \) with \( 0 < \nu < 1 \), 
			\( v \neq \frac{1}{2} \), if and only if \( AX = XB \).
		\end{theorem}
		
		The following result provides a necessary and sufficient condition for the second inequality of \eqref{eq:chain_ineq} to hold.

		\begin{theorem}\label{thm:2nd_eq}
			Let $1<p<\infty$ and let $\mathbf{T}=\mathbf{V}P=(V_1P,\ldots, V_dP)\in\mathfrak{C}^d_p(\mathcal{H})$ be a commuting $d$-tuple. Then  
			\begin{equation}\label{eq:trans_eq}
				\|\widehat{{\bf T}}(t) \|_{s,p}=\|\widehat{{\bf T}}\|_{s,p}  \quad  (t\in (0,1)),
			\end{equation}
			if and only if $\mathbf{T}$ is normal. 
		\end{theorem}
		\begin{proof}
			
			If $\mathbf{T}$ is normal, then it is spherically quasinormal. Using \cite[Lemma 2.1]{CurtoYoon19}, we have that $V_iP=PV_i$, $i\in\{1,\ldots,d\}$. From here, it is obvious that $\widehat{{\bf T}}(t)=\mathbf{T}$ for any $t\in[0,1]$. In particular, $\widehat{{\bf T}}(t)=\widehat{{\bf T}}$, showing that \eqref{eq:trans_eq} holds. 
			
			\medskip 
			
			Now assume that $\eqref{eq:trans_eq}$ is true. It follows from \eqref{eq:transform_t_p} that for any $t\in[0,1]$,
			
			\begin{align*}\label{Eq_01}
				\|\widehat{{\bf T}}(t)\|_{s,p}=\frac{1}{2}\|\mathbb{P}^t\mathbb{V}\mathbb{P}^{1-t}+\mathbb{P}^{1-t}\mathbb{V}\mathbb{P}^t\|_p.
			\end{align*}
			Hence, $\|\widehat{{\bf T}}(t) \|_{s,p}=\|\widehat{{\bf T}}\|_{s,p}$ is equivalent with 
			\begin{equation*}
				\|\mathbb{P}^t\mathbb{V}\mathbb{P}^{1-t}+\mathbb{P}^{1-t}\mathbb{V}\mathbb{P}^t\|_p=\|\mathbb{P}\mathbb{V}+\mathbb{V}\mathbb{P}\|_p.
			\end{equation*}
			
			By Theorem \ref{thm:equiv_cond_p}, we have that $\mathbb{P}\mathbb{V}=\mathbb{V}\mathbb{P}$. Now, Theorem \ref{thm:spher_quasi_matrix} yields the spherical quasinormality of $\mathbf{T}$. Consequently, $\mathbf{T}$ is jointly hyponormal, and in particular, for each $i\in\{1,\ldots,d\}$,  $T_i$ is hyponormal. We also have that  $T_i$ is compact. Since every compact hyponormal operator is in fact normal (see \cite[Corollary 4.9]{Conway91}), we obtain that $T_i$ is normal for each $i\in\{1,\ldots,d\}$, showing that $\mathbf{T}$ is normal. 
		\end{proof}

		In a similar fashion, we consider the first inequality in \eqref{eq:chain_ineq}. Note that due to an invertibility condition, we are restricted to a finite-dimensional case.

		\begin{theorem}
			Let $1<p<\infty$ and let $\mathbf{T}=(T_1,\ldots,T_d)$ be a $d$-tuple of commuting complex matrices. If $\mathbf{T}$ is Taylor invertible, then  
			\begin{equation*}\label{eq:trans_eq_am}
				\|\widetilde{{\bf T}}\|_{s,p}= \|\widehat{{\bf T}}(t)\|_{s,p}\quad  \left(t\in (0,1), t\neq \frac{1}{2}\right),
			\end{equation*}
			if and only if $\mathbf{T}$ is normal. 
		\end{theorem}
		
		\begin{proof}
			Since $\mathbf{T}$ is Taylor invertible, \cite[Corollary 3.6]{Curto81} implies that $P$ is invertible. Consequently, $\mathbb{P}$ is also invertible. The proof now follows from Theorem \ref{thm:equiv_cond_p_1} and by  using the similar lines of arguments as in the proof of Theorem \ref{thm:2nd_eq}.
		\end{proof}
		
		\bigskip

		\section{Joint Schatten $p$-numerical radius inequalities}
		Following the work of \cite{AharmimLabbane24}, we give the following definition:
		\begin{definition}
			Let $1\leq p<\infty$ and $\mathbf{T}=(T_1,\ldots,T_d)\in \mathfrak{C}_p^d(\mathcal{H})$. The (joint) Schatten $p$-numerical radius of $\mathbf{T}$, $\omega_{s,p}(\mathbf{T})$, is defined as
			\begin{equation*}
				\omega_{s,p}(\mathbf{T}):=\sup_{\bm{\lambda}\in \overline{\mathbb{B}_d}}\omega_p\left(\lambda_1 T_1+\cdots+\lambda_dT_d\right),
			\end{equation*}  
			where $\mathbb{B}_d$ is given by \eqref{eq:unit_ball} and $\omega_p(T)=\sup_{\theta\in\mathbb{R}}\norm{\Re\left(e^{i\theta}T\right)}_p$, $T\in\mathfrak{C}_p(\mathcal{H})$.
			In other words, 
			\begin{equation*}
				\omega_{s,p}(\mathbf{T})=\sup_{\bm{\lambda}\in \overline{\mathbb{B}_d}}\sup_{\theta\in\mathbb{R}} \norm{ \Re\left[e^{i\theta}(\lambda_1 T_1+\cdots+\lambda_dT_d)\right]}_p.
			\end{equation*}
		\end{definition}

		We also define the following quantity:
		\begin{definition}
			Let $1\leq p<\infty$ and $\mathbf{T}=(T_1,\ldots,T_d)\in \mathfrak{C}_p^d(\mathcal{H})$. The Schatten hypo-$p$-norm of $\mathbf{T}$, $\norm{\mathbf{T}}_{s,h,p}$, is defined as 
			\begin{equation*}
				\norm{\mathbf{T}}_{s,h,p}:=\sup_{\bm{\lambda}\in\overline{ \mathbb{B}_d}}\norm{ \lambda_1 T_1+\cdots+\lambda_dT_d}_p.
			\end{equation*}
		\end{definition}

		The following theorem shows that the previously introduced notions are well-defined and provides fundamental relations between them.
		
		\begin{theorem}\label{thm:nr_norm_p_ineq}
			Let $1\leq p<\infty$ and $\mathbf{T}=(T_1,\ldots,T_d)=(V_1P,\ldots, V_dP)\in \mathfrak{C}_p^d(\mathcal{H})$. Then
			\begin{equation}\label{eq:nr_norm_p_ineq}
				\omega_{s,p}(\mathbf{T})\leq \norm{\mathbf{T}}_{s,h,p}\leq \|{\bf T}\|_{s,p}.
			\end{equation}
		\end{theorem}

		\begin{proof}
			Let us first show the second inequality in \eqref{eq:nr_norm_p_ineq}. Let $\bm{\lambda}=(\lambda_1,\ldots,\lambda_d)\in\mathbb{B}_d$ be arbitrary and let $U_{\bm{\lambda}}=\sum_{i=1}^d\lambda_iV_i$. Then
			\begin{equation*}
				\lambda_1 T_1+\cdots+\lambda_dT_d=U_{\bm{\lambda}}P.
			\end{equation*}
			Note that \cite[Theorem 17]{Albadawi12} applied to $A_i=\lambda_i I$ and $B_i=V_i$, $i\in\{1,\ldots,d\}$, yields
			\begin{equation*}
				\norm{U_{\bm{\lambda}}}=\norm{\sum_{i=1}^d\lambda_iV_i}\leq \left(\sum_{i=1}^d|\lambda_i|^2\right)^\frac{1}{2}\norm{\sum_{i=1}^d|V_i|^2}^\frac{1}{2}\leq 1,
			\end{equation*}
			since $\bm{\lambda}\in \mathbb{B}_d$ and $\sum_{i=1}^dV_i^*V_i\leq I$. Thus, using Lemma \ref{lem:p_norm_equiv}, we have
			\begin{equation*}
				\norm{\lambda_1 T_1+\cdots+\lambda_dT_d}_p=	\norm{U_{\bm{\lambda}}P}_p\leq \norm{U_{\bm{\lambda}}}\norm{P}_p\leq \|{\bf T}\|_{s,p}.
			\end{equation*}
			Therefore,
			\begin{equation*}
				\norm{\mathbf{T}}_{s,h,p}=\sup_{\bm{\lambda}\in\overline{ \mathbb{B}_d}}\norm{ \lambda_1 T_1+\cdots+\lambda_dT_d}_p\leq \|{\bf T}\|_{s,p},
			\end{equation*}
			Since $\norm{T}_p=\norm{T^*}_p$ for any operator $\mathfrak{C}_p(\mathcal{H})$,	using \cite[Theorem 2]{OmarKittaneh_LAA_2019}, it follows that
			
			\begin{equation*}
				\omega_p\left(\lambda_1 T_1+\cdots+\lambda_dT_d\right)\leq \norm{\lambda_1 T_1+\cdots+\lambda_dT_d}_p
			\end{equation*}
			From here, 
			
			\begin{equation*}
				\sup_{\bm{\lambda}\in \overline{\mathbb{B}_d}}\omega_p\left(\lambda_1 T_1+\cdots+\lambda_dT_d\right)\leq \sup_{\bm{\lambda}\in\overline{ \mathbb{B}_d}}\norm{ \lambda_1 T_1+\cdots+\lambda_dT_d}_p,
			\end{equation*}
			i.e., 
			\begin{equation*}
				\omega_{s,p}(\mathbf{T})\leq \norm{\mathbf{T}}_{s,h,p}.
			\end{equation*}
			This completes the proof.
		\end{proof}

		\begin{remark}
			Let $1\leq p<\infty$. Since $\norm{T}_p=\norm{T^*}_p$ and $\Re(T)=\Re(T^*)$ for any $T\in \mathfrak{C}_p(\mathcal{H})$, it is easy to see that 
			$$\norm{\mathbf{T}}_{s,h,p}=\norm{\mathbf{T}^*}_{s,h,p} \quad \text{ and }\quad \omega_{s,p}(\mathbf{T})=\omega_{s,p}(\mathbf{T}^*)$$ for $\mathbf{T}=(T_1,\ldots,T_d)\in \mathfrak{C}_p^d(\mathcal{H})$. In the case $p=2$, using the commutativity of the trace, we also have that
            
			\begin{equation*}	\norm{\mathbf{T}}_{s,2}^2=\mathrm{tr} \left(\sum_{i=1}^dT_i^*T_i\right) =  \mathrm{tr} \left(\sum_{i=1}^dT_iT_i^*\right)=	\norm{\mathbf{T}^*}_{s,2}^2.
            \end{equation*}
			However, the previous equality does not always hold for $	\|\cdot\|_{s,p}$, $p\neq 2$ (we can take $\mathbf{T}$ as in Example \ref{ex:sharp} below). This complements a similar remark regarding the operator norm and the hypo-norm of operator tuples.  
		\end{remark}

		In order to prove our next result, we need the following useful lemma. It follows from \cite[Lemma 2.1] {Hirzallah_Kittaneh_2008}, by choosing the function $f(t)= t^r$, which is concave (convex) for $0<r< 1$ ($r\geq  1$) on $[0,\infty)$. 


		\begin{lemma}\cite[Lemma 2.1] {Hirzallah_Kittaneh_2008} \label{lem:r_mon}
			Let $A_k\in\mathfrak{B}(\mathcal{H})$, $k\in\{1,\ldots,d\}$, be positive operators. Then, for every unitarily invariant norm
			\begin{equation}\label{lem:r_mon_Inq_1}
				\vertiii{\left(\sum_{k=1}^dA_k\right)^r}\leq \vertiii{\sum_{k=1}^dA_k^r},
			\end{equation}
			for $0<r< 1$ and 
			\begin{equation}\label{lem:r_mon_Inq_2}				\vertiii{\left(\sum_{k=1}^dA_k\right)^r}\leq d^{r-1}\vertiii{\sum_{k=1}^dA_k^r},
			\end{equation}
			for $r\geq 1$.
		\end{lemma}

		\begin{theorem}
			Let $1\leq p<\infty$ and $\mathbf{T}=(T_1,\ldots,T_d)\in \mathfrak{C}_p^d(\mathcal{H})$. Then	
			
			\begin{equation*}
				\frac{1}{2}\|{\bf T}\|_{s,h,p}	\leq \omega_{s,p}(\mathbf{T})\leq \|{\bf T}\|_{s,h,p}.
			\end{equation*}
			Moreover, 
			
			\begin{equation}\label{eq:norm_s_p_equiv}
				\frac{1}{\sqrt[p]{d}}\|{\bf T}\|_{s,p}	\leq \|{\bf T}\|_{s,h,p}\leq \|{\bf T}\|_{s,p} 
			\end{equation}
 for $1\leq p< 2$,  and
				\begin{equation}\label{eq:norm_s_p_equiv_2}
					\frac{1}{\sqrt{d}}\|{\bf T}\|_{s,p}	\leq \|{\bf T}\|_{s,h,p}\leq \|{\bf T}\|_{s,p} 
				\end{equation}
 for $2\leq p<\infty$.
		\end{theorem}
		
		\begin{proof}
			Let $\bm{\lambda}=(\lambda_1,\ldots,\lambda_d)\in\mathbb{B}_d$ be arbitrary. Using \cite[Theorem 2]{OmarKittaneh_LAA_2019} again, we have that
			\begin{equation*}
				\frac{1}{2}\norm{\lambda_1 T_1+\cdots+\lambda_dT_d}_p\leq \omega_p\left(\lambda_1 T_1+\cdots+\lambda_dT_d\right)
			\end{equation*}
			and therefore,
			\begin{equation*}
				\frac{1}{2}\norm{\mathbf{T}}_{s,h,p}=\frac{1}{2}\sup_{\bm{\lambda}\in\overline{ \mathbb{B}_d}}\norm{ \lambda_1 T_1+\cdots+\lambda_dT_d}_p\leq \sup_{\bm{\lambda}\in \overline{\mathbb{B}_d}}\omega_p\left(\lambda_1 T_1+\cdots+\lambda_dT_d\right)=\omega_{s,p}(\mathbf{T}).
			\end{equation*}
			Now assume that $1\leq p<2$. Note that for each $i\in\{1,\ldots,d\}$, we have that
			\begin{equation*}
				\norm{\mathbf{T}}_{s,h,p}=\sup_{\bm{\lambda}\in\overline{ \mathbb{B}_d}}\norm{ \lambda_1 T_1+\cdots+\lambda_dT_d}_p\geq \norm{T_i}_p,
			\end{equation*}
			and thus,
			\begin{equation}\label{eq:d_sum}
				d\norm{\mathbf{T}}_{s,h,p}^p\geq \sum_{i=1}^d\norm{T_i}_p^p.
			\end{equation}
			Now, since $\frac{1}{2}\leq \frac{p}{2}< 1$, inequality \eqref{lem:r_mon_Inq_1} implies that 
			\begin{align*}
				\norm{P}_p^p&=	\norm{\left(\sum_{i=1}^{d}T_i^*T_i\right)^\frac{p}{2}}_1\leq \norm{ \sum_{i=1}^{d}(T_i^*T_i )^\frac{p}{2}}_1\\
				&\leq  \sum_{i=1}^{d}\|(T_i^*T_i )^\frac{p}{2}\|_1=\sum_{i=1}^{d}\||T_i|^p\|_1\\
				&=\sum_{i=1}^{d}\|T_i\|^p_p.
			\end{align*}
			It now follows from Lemma \ref{lem:p_norm_equiv} and \eqref{eq:d_sum} that
			\begin{equation*}
				d\norm{\mathbf{T}}_{s,h,p}^p\geq \|{\bf T}\|_{s,p}^p,
			\end{equation*}
			i.e.,
			\begin{equation*}
				\frac{1}{\sqrt[p]{d}}\|{\bf T}\|_{s,p}	\leq \|{\bf T}\|_{s,h,p}.
			\end{equation*}
 
				To prove the complementary inequality, assume that $ 2\leq p<\infty$. 
				Now, since $ \frac{p}{2}\geq 1$, inequality \eqref{lem:r_mon_Inq_2} implies that 
				\begin{align*}
					\norm{P}_p^p&=	\norm{\left(\sum_{i=1}^{d}T_i^*T_i\right)^\frac{p}{2}}_1\leq d^{\frac{p}{2}-1}\norm{ \sum_{i=1}^{d}(T_i^*T_i )^\frac{p}{2}}_1\\
					&\leq d^{\frac{p}{2}-1} \sum_{i=1}^{d}\|(T_i^*T_i )^\frac{p}{2}\|_1=d^{\frac{p}{2}-1}\sum_{i=1}^{d}\||T_i|^p\|_1\\
					&=d^{\frac{p}{2}-1}\sum_{i=1}^{d}\|T_i\|^p_p.
				\end{align*}
				From here,
				\begin{equation*}				d\norm{\mathbf{T}}_{s,h,p}^p\geq\frac{1}{d^{\frac{p}{2}-1}} \|{\bf T}\|_{s,p}^p,
				\end{equation*}
				i.e.,
				\begin{equation*}
					\frac{1}{d^{\frac{p}{2}}}\|{\bf T}\|_{s,p}^p	\leq \|{\bf T}\|_{s,h,p}^p.
				\end{equation*}
 This completes the proof.
		\end{proof}

		\begin{corollary}
			Let $1\leq p<\infty$ and $\mathbf{T}=(T_1,\ldots,T_d)\in \mathfrak{C}_p^d(\mathcal{H})$. Then
			\begin{equation*}
				\frac{1}{2\sqrt[p]{d}}\|{\bf T}\|_{s,p}\leq \frac{1}{2}\|{\bf T}\|_{s,h,p}	\leq \omega_{s,p}(\mathbf{T})\leq \|{\bf T}\|_{s,h,p}\leq \|{\bf T}\|_{s,p} 
			\end{equation*}
			for $1\leq p<2$, and 
			\begin{equation*}
				\frac{1}{2\sqrt{d}}\|{\bf T}\|_{s,p}\leq \frac{1}{2}\|{\bf T}\|_{s,h,p}	\leq \omega_{s,p}(\mathbf{T})\leq \|{\bf T}\|_{s,h,p}\leq \|{\bf T}\|_{s,p} 
			\end{equation*}
			for $2\leq p<\infty$.
		\end{corollary}
		
		By using \cite[Theorem 8]{OmarKittaneh_LAA_2019} and a similar argumentation as in the proof of Theorem \ref{thm:nr_norm_p_ineq}, we have that the following better estimate holds in the case $p=2$.
		
		\begin{corollary}
			Let $\mathbf{T}=(T_1,\ldots,T_d)\in \mathfrak{C}_2^d(\mathcal{H})$. Then
			\begin{equation*}
				\frac{1}{\sqrt{2d}}\|{\bf T}\|_{s,2}\leq \frac{1}{\sqrt{2}}\|{\bf T}\|_{s,h,2}	\leq \omega_{s,2}(\mathbf{T})\leq \|{\bf T}\|_{s,h,2}\leq \|{\bf T}\|_{s,2}.
			\end{equation*}
		\end{corollary}

		\begin{example}\label{ex:sharp}
			Note that the first inequality in \eqref{eq:norm_s_p_equiv_2} is sharp. Indeed, let $p\geq 2$ be arbitrary and let $\mathbf{T}=(T_1,T_2)$, where $T_1=\begin{bmatrix}
				1&0\\0&0
			\end{bmatrix}$ and $T_2=\begin{bmatrix}
				0&0\\1&0
			\end{bmatrix}$. Then
			\begin{equation*}
				P=\sqrt{T_1^*T_1+T_2^*T_2}=\begin{bmatrix}
					\sqrt{2}&0\\0&0
				\end{bmatrix},
			\end{equation*}
			and thus, 
			\begin{equation*}
				\|{\bf T}\|_{s,p}^p=\mathrm{tr}\,(P^p)=2^\frac{p}{2}.
			\end{equation*}
			On the other hand,
			\begin{align*}
				\|{\bf T}\|_{s,h,p}^p&=\sup_{|\lambda_1|^2+|\lambda_2|^2=1}\norm{\lambda_1 T_1+\lambda_2 T_2}_p^p=\sup_{|\lambda_1|^2+|\lambda_2|^2=1}\norm{\begin{bmatrix}
						\lambda_1&0\\\lambda_2&0
				\end{bmatrix}}_p^p\\
				&=\sup_{|\lambda_1|^2+|\lambda_2|^2=1}\mathrm{tr}\,\left(\begin{bmatrix}
					\sqrt{|\lambda_1|^2+|\lambda_2|^2}&0\\0&0\\
				\end{bmatrix}^p\right)\\
				&=\sup_{|\lambda_1|^2+|\lambda_2|^2=1}\left(|\lambda_1|^2+|\lambda_2|^2\right)^\frac{p}{2}=1.
			\end{align*}
			Therefore,
			\begin{equation*}
				\frac{1}{\sqrt{2}}\|{\bf T}\|_{s,p}=1=\|{\bf T}\|_{s,h,p}.
			\end{equation*}
		\end{example}
		
		\begin{example}
			Let $p\geq 1$ and let  $\mathbf{T}=(T_1,T_2)$, where $T_1=\begin{bmatrix}
				1&0\\0&0
			\end{bmatrix}$ and $T_2=\begin{bmatrix}
				0&0\\0&1
			\end{bmatrix}$. Then
			\begin{equation*}
				P=\sqrt{T_1^*T_1+T_2^*T_2}=I,
			\end{equation*}
			and therefore,
			\begin{equation*}
				\frac{1}{\sqrt[p]{2}}\|{\bf T}\|_{s,p}=\frac{1}{2^\frac{1}{p}}\left[\mathrm{tr}\,(P^p)\right]^\frac{1}{p}=\frac{1}{2^\frac{1}{p}}\cdot2^\frac{1}{p}=1.
			\end{equation*}
			On the other hand,
			\begin{align*}
				\|{\bf T}\|_{s,h,p}&=\sup_{|\lambda_1|^2+|\lambda_2|^2=1}\norm{\lambda_1 T_1+\lambda_2 T_2}_p=\sup_{|\lambda_1|^2+|\lambda_2|^2=1}\norm{\begin{bmatrix}
						\lambda_1&0\\0&\lambda_2
				\end{bmatrix}}_p\\
				&=\sup_{|\lambda_1|^2+|\lambda_2|^2=1}\left(|\lambda_1|^p+|\lambda_2|^p\right)^\frac{1}{p}=1.
			\end{align*}
			This shows that the first inequality in \eqref{eq:norm_s_p_equiv} is sharp and the equality can hold for any $p\geq 1$.
		\end{example}

		\section*{Declarations}
		\noindent{\bf{Funding}}\\
		This work has been supported by the Ministry of Science, Technological Development and Innovation of the Republic of Serbia [Grant Number: 451-03-137/2025-03/200102].	
		
		\vspace{0.5cm}
		
		\noindent{\bf{Availability of data and materials}}\\
		\noindent No data were used to support this study.
		\vspace{0.5cm}\\
		\noindent{\bf{Competing interests}}\\
		\noindent The author declares that he has no competing interests.
		\vspace{0.5cm}
		
		\noindent{\bf{Authors' contributions}}\\
		\noindent
		The work was a collaborative effort of all authors, who contributed equally to writing the article. All authors have read and approved the final manuscript.
		
		\vspace{0.5cm}
		
		
		
	\end{document}